\documentclass[11pt]{amsart}
\usepackage{amscd,amssymb,verbatim}
\theoremstyle{plain}
\newtheorem{Thm}{Theorem}[section]

\newtheorem{Lem}[Thm]{Lemma}

\theoremstyle{definition}
\newtheorem{remark}{Remark}

\numberwithin{equation}{section}

\begin{document}
\title[A proof of Rosenthal's \(\ell_1\)-Theorem]{A proof of Rosenthal's \(\ell_1\)-Theorem}
\author{I. Gasparis}
\address{School of Applied Mathematical and Physical Sciences \\
National Technical University of Athens\\
Athens 15780 \\
Greece}
\email{ioagaspa@math.ntua.gr}
\begin{abstract}
A bounded sequence \((x_n)\) in a complex Banach space is called \(\epsilon\)-weak Cauchy, for
some \(\epsilon > 0\), if for all \(x^* \in B_{X^*}\) there exists some \(n_0 \in \mathbb{N}\)
such that \(|x^*(x_n) - x^*(x_m)| < \epsilon\) for all \(n \geq n_0\) and \(m \geq n_0\).
It is shown that given \(\epsilon > 0\) and a bounded sequence \((x_n)\) in a Banach space then either
\((x_n)\) admits an \(\epsilon\)-weak Cauchy subsequence or, for all \(\delta > 0\), there exists a subsequence \((x_{m_n})\)
satisfying 
\[\biggl \| \sum_{n=2}^\infty \lambda_n (x_{m_n} - x_{m_1}) \biggr \| \geq (\frac{\epsilon \sqrt{2}}{8} - \delta )
\sum_{n=2}^\infty (|\mathrm{Re} (\lambda_n)| + |\mathrm{Im} (\lambda_n)|)\]
for every finitely supported sequence of complex scalars \((\lambda_n)_{n=2}^\infty\).
This provides an alternative proof for both Rosenthal's \(\ell_1\)-theorem and its quantitative version
due to Behrends.
\end{abstract}
\maketitle
\section{Introduction}
The purpose of this article is to provide an alternative proof of Rosenthal's
\(\ell_1\)-theorem \cite{R}
\begin{Thm} \label{Tm}
Every bounded sequence in a (real or complex) Banach space admits a 
subsequence which is either weak Cauchy, or equivalent to the usual \(\ell_1\)-basis.
\end{Thm}
Rosenthal proved his result for real Banach spaces and subsequently Dor \cite{D}
settled the complex case. Most proofs of the \(\ell_1\)-theorem (\cite{F}, \cite{O},
\cite{B}, \cite{G}) rely on the infinite Ramsey theorem \cite{E} which states that every analytic subset
of \([\mathbb{N}]\) is Ramsey (see also \cite{GP}). A proof of the \(\ell_1\)-theorem
that avoids the use of Ramsey theory is given in \cite{A}.
We recall here that for 
an infinite subset \(L\) of \(\mathbb{N}\), \([L]\) stands for
the set of all of its infinite subsets. \([\mathbb{N}]\) is endowed with the topology of
point-wise convergence. A subset \(\mathcal{A}\) of \([\mathbb{N}]\)
is a Ramsey set if for every \(N \in [\mathbb{N}]\) there exists \(M \in [N]\) 
such that either \([M] \subset \mathcal{A}\), or \([M] \cap \mathcal{A} = \emptyset\).

To explain our proof of the \(\ell_1\)-theorem, we first
fix a compact metric space \(K\) and a bounded sequence \((f_n)\) of complex-valued
functions, continuous on \(K\). Given \(\epsilon > 0\), let us call 
\((f_n)\) \(\epsilon\)-{\em weak Cauchy} provided that for every \(t \in K\)
there exists \(n_0 \in \mathbb{N}\) such that \(|f_n(t) - f_m(t)| < \epsilon\)
for all \(n \geq n_0\) and \(m \geq n_0\). We note that the concept of an \(\epsilon\)-weak Cauchy
sequence is implicit in \cite{B}. Our main result is as follows
\begin{Thm} \label{MTh}
Let \((f_n)\) be a bounded sequence in a complex \(C(K)\) space and let \(\epsilon > 0\).
Then either \((f_n)\) admits an \(\epsilon\)-weak Cauchy subsequence or, for every \(\delta > 0\), there is 
a subsequence \((f_{m_n})\) satisfying 
\[\biggl \|\sum_{n=2}^\infty \lambda_n (f_{m_n} - f_{m_1}) \biggr \| \geq \bigl (\frac{\epsilon \sqrt{2}}{8} - \delta \bigr ) 
\sum_{n=2}^\infty (|\mathrm{Re} (\lambda_n)| + |\mathrm{Im} (\lambda_n)|)\]
for every finitely supported sequence of complex scalars \((\lambda_n)_{n=2}^\infty\).
\end{Thm} 
If the second alternative of Theorem \ref{MTh} holds, then \((f_{m_n} - f_{m_1})_{n=2}^\infty\)
\(C\)-dominates the usual \(\ell_1\)-basis, where \(C = \frac{\epsilon \sqrt{2}}{8} - \delta \).
It follows now, by a result of H. Knaust and E. Odell (Proposition 4.2 of \cite{KO} which holds for complex Banach spaces as well),
that there is some \(k \geq 2\) so that \((f_{m_n})_{n=k}^\infty\) \(C\)-dominates the usual \(\ell_1\)-basis. 
Therefore, if \((f_n)\) admits no \(\ell_1\)-subsequence, then every subsequence of \((f_n)\)
admits, for all \(\epsilon > 0\), an \(\epsilon\)-weak Cauchy subsequence and thus
a weak Cauchy subsequence. On the other hand, if \((f_n)\) lacks \(\epsilon\)-weak Cauchy subsequences for some fixed \(\epsilon > 0\),
then Theorem \ref{MTh} combined with the above cited result from \cite{KO}, yields the quantitative version of the \(\ell_1\)-theorem obtained by E. Behrends in \cite{B}.

We use standard Banach space facts and terminology as may be found in \cite{LT}.

{\bf Notation}: Let \(E\) be a closed subset of \(\mathbb{C}^2\) and \(P \in [\mathbb{N}]\).
Define
\[\mathcal{D}(P,E) = \{ L \in [P], \, L = (l_n) : \, \exists \, t \in K, \,
(f_{l_{2n-1}} (t), f_{l_{2n}} (t)) \in E, \, \forall \, n \in \mathbb{N} \}\] 
{\bf Notation}: For \(\epsilon > 0\) we set
\(F_\epsilon = \{(z_1, z_2 ) \in \mathbb{C}^2 : \,
|z_1 - z_2 | \geq \epsilon\}\).
\begin{Lem} \label{L1}
\(\mathcal{D}(P,E)\) is point-wise closed in \([\mathbb{N}]\).
\end{Lem}
\begin{proof}
Let \((L_m)\) be a sequence in \(\mathcal{D}(P,F)\) converging point-wise to
some \(L \in [P]\), \(L=(l_n)\). Fix \(n \in \mathbb{N}\) and choose \(m_n \in \mathbb{N}\)
such that \(\{l_k : \, k \leq 2n \}\) is an initial segment of \(L_{m_n}\).
Next choose \(t_n \in K\) so that  
\((f_{l_{2j-1}} (t_n), f_{l_{2j}} (t_n)) \in E\) for all \(j \leq n\).
Let \(t \in K\) be a cluster point of the sequence \((t_n)\). Since the
\(f_j\)'s are continuous and \(E\) is closed, we obtain that
\((f_{l_{2n-1}} (t), f_{l_{2n}} (t)) \in E\) for all \(n \in \mathbb{N}\)
and so \(L \in \mathcal{D}(P,E)\)
completing the proof of the lemma.
\end{proof}
\begin{remark}
Assume that \(\mathcal{D}(P,F_\epsilon)\)
is a proper subset of \([P]\) for all \(P \in [\mathbb{N}]\) and \(\epsilon > 0\).
Since \(\mathcal{D}(P,F_\epsilon)\) is Ramsey, an easy diagonalization argument
shows that \((f_n)\) admits a weak Cauchy subsequence.
\end{remark}
\begin{Lem} \label{L3}
Assume that \([\mathbb{N}] = \mathcal{D}(\mathbb{N},F_\epsilon)\) for some
\(\epsilon > 0\). Let \(\delta_1 > 0\) and choose an integer multiple \(A\) of \(\delta_1\), \(A > \delta_1\), so that 
\(\|f_n\| \leq A\) for all \(n \in \mathbb{N}\).
Then there exist \(M \in [\mathbb{N}]\) and
two squares \(\Delta_1\) and \(\Delta_2\) contained in \([-A,A]^2\) with sides
parallel to the axes and equal to \(\delta_1\) so that
\begin{enumerate}
\item \((\Delta_1 \times \Delta_2) \cap F_\epsilon \ne \emptyset\).
\item For every choice \(J_1\) and \(J_2\) of pairwise disjoint subsets of \(M\)
there exists \(t \in K\) so that \(f_n(t) \in \Delta_s\) for all \(n \in J_s\) and \(s \leq 2\).
\end{enumerate}
\end{Lem} 
\begin{proof}
Let \(\Pi\) be a finite partition of \([-A,A]^2\) into pairwise non-overlapping squares
having sides parallel to the axes and equal to \(\delta_1\). Define
\[\mathcal{D} = \cup_{(\Delta_1, \Delta_2) \in \Pi^2} \, \mathcal{D}[\mathbb{N}, 
(\Delta_1 \times \Delta_2) \cap F_\epsilon]\]
This is a point-wise closed subset of \([\mathbb{N}]\), thanks to Lemma \ref{L1}, and therefore
it is Ramsey. Let \(L \in [\mathbb{N}]\), \(L=(l_n)\). Since \(L \in \mathcal{D}(\mathbb{N},F_\epsilon)\)
there exist 
\(t \in K\), \(I \in [\mathbb{N}]\) and \((\Delta_1, \Delta_2) \in \Pi^2\)
so that \((f_{l_{2j-1}}(t), f_{l_{2j}}(t)) \in (\Delta_1 \times \Delta_2) \cap F_\epsilon \)
for all \(j \in I\). It follows that \(L_1 \in [L] \cap \mathcal{D}\), where 
\(L_1 = \cup_{j \in I} \{l_{2j-1}, l_{2j}\}\). Hence, \([L] \cap \mathcal{D} \ne \emptyset\)
for all \(L \in [\mathbb{N}]\). The infinite Ramsey Theorem now yields \(N \in [\mathbb{N}]\)
so that \([N] \subset \mathcal{D}\). A second application of the infinite Ramsey Theorem provides
us some \(P \in [N]\) and \((\Delta_1, \Delta_2) \in \Pi^2\) so that
\([P] \subset \mathcal{D}[\mathbb{N}, 
(\Delta_1 \times \Delta_2) \cap F_\epsilon]\). Finally, if \(P=(p_n)\),
let \(M= \{p_{3n-1}: \, n \in \mathbb{N}\}\). Clearly, \(M\), \(\Delta_1\) and \(\Delta_2\)
satisfy \((1)\) and \((2)\).  
\end{proof}
\begin{Lem} \label{L4}
Let \(\epsilon > 0\), \(\delta > 0\) and
\( 0 < \delta_1 <  \delta /5\).
Assume that \((1)\) and \((2)\) of Lemma \ref{L3} are fulfilled with \(M = \mathbb{N}\),
\(\Delta_1 = [0, \delta_1]^2\)
and \(\Delta_2\) contained in the first quadrant.
Then \[ \biggl \|\sum_{n \geq 2} \lambda_n (f_n - f_1) \biggr \| \geq 
(\frac{\epsilon \sqrt{2}}{8} - \delta) \sum_{n \geq 2} ( | \mathrm{ Re } (\lambda_n) |  + | \mathrm{ I m } (\lambda_n) |) \]
for every finitely supported sequence of complex scalars \((\lambda_n)_{ n \geq 2}\). 
\end{Lem}
\begin{proof}
Since \((\Delta_1 \times \Delta_2) \cap F_\epsilon \ne \emptyset\) there exists
\(z_2 \in \Delta_2\) so that \(|z_2| \geq \epsilon - \delta_1 \sqrt{2}\).
Let \((\lambda_n)_{ n \geq 2}\) be a finitely supported sequence
of complex scalars and set \(a_n = \mathrm{Re}(\lambda_n)\), \(b_n = \mathrm{Im}(\lambda_n)\),
for all \(n \geq 2\). 
We next choose a finite subset
\(J_2 \) of \(\{ n \in \mathbb{N} : \, n \geq 2 \}\) so that each one of the sequences \((a_n)_{n \in J_2}\) and
\((b_n)_{n \in J_2}\) consists of terms with equal signs and 
\[\sum_{n \in J_2} (|a_n| + |b_n|) \geq (1/4) \sum_{n \geq 2} (|a_n| + |b_n|)\] 
We may suppose that \(a_n b_n \geq 0\) for all \(n \in J_2\). The case 
\(a_n b_n \leq 0\) for all \(n \in J_2\) is handled with similar arguments.
Let
\(J_1\) be the complement of \(J_2\) in the support of \((\lambda_n)_{n \geq 2}\). By our hypothesis, there is
some \(t \in K\) such that \(f_n(t) \in \Delta_2\) for all \(n \in J_2\), while 
\(f_n(t) \in \Delta_1\) for all \(n \in J_1 \cup \{1\}\).
Note that
\begin{align} \label{E1}
\biggl \| \sum_{n \geq 2} \lambda_n (f_n - f_1) \biggr \| &\geq \biggl | \sum_{n \in J_2} \lambda_n f_n(t) \biggr | -
\sum_{n \in J_1} |\lambda_n| |f_n(t)|  - \sum_{n \geq 2} |\lambda_n | |f_1(t) | \\
&\geq \biggl | \sum_{n \in J_2} \lambda_n f_n (t) \biggr | - 2 \delta_1 \sqrt{2} \sum_{n \geq 2} |\lambda_n|
\notag \\
&\geq \biggl | \sum_{n \in J_2} \lambda_n z_2 \biggr | - \sum_{n \in J_2} |\lambda_n| |f_n(t) - z_2| 
- 2 \delta_1 \sqrt{2} \sum_{n \geq 2} |\lambda_n| \notag \\
&\geq \biggl | \sum_{n \in J_2} \lambda_n z_2 \biggr | - 3 \delta_1 \sqrt{2} \sum_{n \geq 2} |\lambda_n| \notag
\end{align}
We next write \(z_2 = a + i b\) where \(a\) and \(b\) are non-negative reals and estimate 
\begin{align} \label{E2}
\biggl | \sum_{n \in J_2} \lambda_n z_2 \biggr |^2 &=  
\biggl | \sum_{n \in J_2} (a_n a - b_n b ) \biggr |^2  +
\biggl | \sum_{n \in J_2} (a_n b + b_n a ) \biggr |^2  \\
&= \biggl | a \sum_{n \in J_2}  |a_n|  - b \sum_{n \in J_2} |b_n|   \biggr |^2  +
\biggl | b \sum_{n \in J_2} |a_n|  + a \sum_{n \in J_2} |b_n|  \biggr |^2 \notag \\
&\geq \frac{a^2 + b^2}{2} \biggl (\sum_{n \in J_2} (|a_n| +|b_n|) \biggr )^2 
\geq (|z_2|^2 /32) \biggl (\sum_{n \geq 2} (|a_n| +|b_n| ) \biggr )^2 \notag
\end{align}
because \((ax - by)^2 + (bx + ay)^2 \geq (a^2 + b^2)/2\), whenever \(x\), \(y\) are non-negative scalars
with \(x +y =1\).
Since \(|z_2| \geq \epsilon - \sqrt{2} \delta_1\), it follows from (\ref{E1}), (\ref{E2}) that
\[\biggl \| \sum_{n \geq 2} \lambda_n (f_n - f_1) \biggr \| \geq (\frac{\epsilon - \delta_1 \sqrt{2}}{4 \sqrt{2}} - 3 \delta_1 \sqrt{2} ) 
\sum_{n \geq 2} ( | \mathrm{ Re } (\lambda_n) |  + | \mathrm{ I m } (\lambda_n) |)\]
which yields the assertion of the lemma.
\end{proof}
\begin{proof}[Proof of Theorem \ref{MTh}]
Suppose that \((f_n)\) has no subsequence which is \(\epsilon\)-weak Cauchy. 
It follows that \([P] \cap \mathcal{D}(\mathbb{N}, F_\epsilon) = \emptyset\)
for no \(P \in [\mathbb{N}]\). Since \(\mathcal{D}(\mathbb{N}, F_\epsilon)\)
is point-wise closed, by Lemma \ref{L1}, the infinite Ramsey theorem implies that
\(\mathcal{D}(N, F_\epsilon) = [N]\) for some \(N \in [\mathbb{N}]\). Without loss of generality, by
relabeling if necessary, we may assume that \(N = \mathbb{N}\). Let \(\delta > 0\) and choose 
\(0 < \delta_1 < \delta /5\).
Applying Lemma \ref{L3},
passing to a subsequence and relabeling if necessary, we may also assume that there exist
two squares \(\Delta_1\) and \(\Delta_2\) with sides
parallel to the axes and equal to \(\delta_1\) so that
\begin{enumerate}
\item \((\Delta_1 \times \Delta_2) \cap F_\epsilon \ne \emptyset\).
\item For every choice \(J_1\) and \(J_2\) of pairwise disjoint subsets of \(\mathbb{N}\)
there exists \(t \in K\) so that \(f_n(t) \in \Delta_s\) for all \(n \in J_s\) and \(s \leq 2\).
\end{enumerate}
By replacing \((f_n)\) by \((g_n)\), where for all \(n \in \mathbb{N}\)
\(g_n = \sigma f_n + z\) for suitable choices
of \(\sigma \in \{1,i\}\) and \(z \in \mathbb{C}\), we may assume that
\(\Delta_1 = [0, \delta_1]^2\)
and that \(\Delta_2\) is contained in the first quadrant. The assertion of the theorem now follows by applying Lemma 
\ref{L4}.
\end{proof}
A direct modification of the preceding arguments shows that theorem \ref{MTh} holds for real \(C(K)\) spaces with the constant \(\sqrt{2}/8\)
being replaced by \(1/2\).

\end{document}